\documentclass[reqno]{amsart}

\usepackage{graphicx}
\usepackage{amscd}
\usepackage{amsmath, amssymb}
\usepackage{graphics}
\usepackage{graphicx}
\usepackage{epsfig}
\usepackage{xcolor}
\usepackage{soul}
\usepackage{tikz}

\newtheorem{theorem}{Theorem}

\newtheorem{corollary}[theorem]{Corollary}

\newtheorem{definition}[theorem]{Definition}

\newtheorem{lemma}[theorem]{Lemma}

\newtheorem{proposition}[theorem]{Proposition}
\newtheorem{remark}[theorem]{Remark}

\newtheorem{question}[theorem]{Question}

\DeclareMathOperator{\id}{Id}

\begin{document}

\title[]{On the structure of the singular triplet monoid and its virtual extension}

\author{Carmen Caprau}
\address{Carmen Caprau\\
               Department of Mathematics\\
                California State University, Fresno\\
                5245 N. Backer Ave. M/S PB 108\\
                Fresno, ca 93740, USA}

\email{ccaprau@csufresno.edu}

\author{Mohamad N. Nasser}
\address{Mohamad N. Nasser\\
         Department of Mathematics and Computer Science\\
         Beirut Arab University\\
         P.O. Box 11-5020, Beirut, Lebanon}
         
\email{m.nasser@bau.edu.lb}

\maketitle
 
\begin{abstract}
In this article, we introduce two new algebraic structures associated with the triplet group on $n$ strands, $L_n$: the singular triplet monoid $SLM_n$ and its virtual extension $VSLM_n$, defined in analogy with the singular braid monoid and the virtual singular braid monoid. We begin by presenting these monoids in terms of generators and relations, and then derive several alternative presentations of $VSLM_n$. Second, we investigate the problem of extending representations of $L_n$ to these monoids. Two extension methods are developed: the $k$-local type extension, which applies to $k$-local representations, and the $\Phi$-type extension, which applies to representations satisfying suitable commutativity conditions. We show that every $2$-local representation of $L_n$ admits extensions to both $SLM_n$ and $VSLM_n$ via the two methods. As an application, we consider a specific representation $\mu : L_n \longrightarrow \mathrm{GL}_n(\mathbb{Z}[t^{\pm1}])$ introduced recently by Nasser et al. We explicitly determine all homogeneous $2$-local extensions of $\mu$ to $SLM_n$ and $VSLM_n$, and compute the corresponding $\Phi$-type extensions. Furthermore, we compare these two extension methods, showing that they coincide for $SLM_n$ under suitable parameter conditions, while they do not coincide for $VSLM_n$. These results provide a systematic framework for extending representations of $L_n$ to $SLM_n$ and $VSLM_n$.
\end{abstract}

\renewcommand{\thefootnote}{}

\vspace*{0.2cm}

\section{Introduction}

Coxeter groups form a fundamental and widely studied class of groups in both algebra and geometry, naturally arising from reflections, symmetry considerations, and combinatorial patterns. Formally, a Coxeter group $\mathcal{C}$ can be presented as
\[
\mathcal{C} = \langle c_1, c_2, \dots, c_s \mid c_i^2 = 1, \ (c_i c_j)^{m_{ij}} = 1, \ 1 \le i,j \le s \rangle,
\]
where $m_{ij} = 1$ if $i=j$ and $m_{ij} \ge 2$ when $i \neq j$. A classical example is the symmetric group on $n$ elements, denoted by $S_n$, whose elements represent all possible permutations of $n$ objects, and whose Coxeter presentation is given by involutions $\alpha_1, \alpha_2, \dots, \alpha_{n-1}$ satisfying standard braid-like and commutation relations.

Beyond the symmetric group $S_n$, several generalizations incorporating additional combinatorial features were introduced in \cite{MKhovanov1996,Khovanov1997}. One such example is the twin group on $n$ strands denoted by $T_n$. It is generated by elements $s_1,s_2,\ldots,s_{n-1}$ satisfying the relations $s_i^2=1$ for all $1\leq i\leq n-1$ and $s_is_j=s_js_i$ whenever $|i-j|\geq2$. Another related construction is the triplet group on $n$ strands, denoted by $L_n$, which is generated by elements $\ell_1,\ell_2,\ldots,\ell_{n-1}$ subject to the relations $\ell_i^2=1$ for all $1\le i\le n-1$ and $\ell_i\ell_{i+1}\ell_i=\ell_{i+1}\ell_i\ell_{i+1}$ for all $1\le i\le n-2$. These algebraic structures also admit natural topological interpretations. Twin groups are closely related to doodles on the $2$–sphere, which consist of finite collections of simple closed curves on the sphere with no triple or higher–order intersections. In a similar spirit, triplet groups are associated with topological objects known as noodles. Fixing a codimension–one foliation with singularities on the $2$–sphere, a noodle is defined as a collection of closed curves such that no two intersection points lie on the same leaf of the foliation, quadruple intersections are excluded, and no intersection point coincides with a singular point of the foliation.

The braid group on $n$ strands, $B_n$, originally introduced in \cite{EArtin1926, EArtin1947}, occupies a central role in both algebra and topology. Generated by the elementary braids $\sigma_1, \sigma_2, \dots, \sigma_{n-1}$, $B_n$ models the ways in which $n$ strands can intertwine without breaking. Its algebraic study, particularly through linear representations, has deep connections to knot theory, mapping class groups, and configuration spaces. Braid groups are closely linked to Coxeter groups: $B_n$ can be viewed as a “braid-like” lift of the symmetric group $S_n$, with the latter appearing as a quotient under the natural projection $\sigma_i \longmapsto (i\ \ i+1)$. Nevertheless, $B_n$ differs fundamentally from Coxeter groups due to its non-involutive generators and richer topological structure, which enables its broad applications in physics, cryptography, and quantum invariants.

Over the years, several extensions of the braid group $B_n$ have been introduced to model more subtle topological phenomena. Notable examples include the singular braid monoid $SM_n$ and the singular braid group $SB_n$ \cite{baezz, birman, Fenn}, as well as their virtual counterparts, the virtual singular braid monoid $VSM_n$ and the virtual singular braid group $VSB_n$ \cite{Caprau, Caprau1, Caprau2}, which account for crossings with singularities or virtual interactions. In a similar spirit, the twin group $T_n$ has been extended to incorporate singular and virtual singular versions, denoted $STM_n$ and $ST_n$ for the singular case, and $VSTM_n$ and $VST_n$ for the virtual singular case \cite{NasserNafaa, CaprauNasser}. These constructions illustrate the importance of studying representations beyond the classical groups. This naturally leads to the question: Can analogous singular and virtual extensions be constructed for the triplet group $L_n$? Developing such extensions and their representations is the primary motivation for our work. 

In this work, we focus on the triplet group $L_n$ and investigate the construction of several related extensions, namely the singular triplet monoid $SLM_n$, the singular triplet group $SL_n$, and their virtual extensions $VSLM_n$ and $VSL_n$. The main objective here is to develop several presentations for these structures and to investigate the problem of extending representations of $L_n$ to them. In particular, we introduce systematic methods for constructing such extensions, thereby providing a unified framework that relates $L_n$ to its extensions, particularly $SLM_n$ and $VSLM_n$.

Our results are as follows. In Sections 2 and 3, we introduce the singular triplet monoid $SLM_n$ and its virtual extension $VSLM_n$, together with their corresponding groups $SL_n$ and $VSL_n$. In addition, we derive several alternative presentations for the virtual singular triplet monoid $VSLM_n$. In Section 4, we investigate the problem of extending representations of $L_n$ to $SLM_n$ and $VSLM_n$. More deeply, we discuss two types of extension: the $k$-local extension and the $\Phi$-type extension. In Section 5, we extend a particular representation $\mu$ of $L_n$ to $SLM_n$ and $VSLM_n$ using both extension methods and study the relationships between the resulting representations. Finally, in Section 6, we present several directions for future research aimed at further developing and generalizing this approach.

\section{The Singular Triplet Monoid and Its Virtual Extension}

\subsection{The Braid Group and Its Singular Extensions}

To introduce the singular triplet monoid and its virtual extension, we begin by recalling the fundamental definitions of the braid group, the singular braid monoid, and the virtual singular braid monoid. These algebraic structures play a central role in braid theory and provide the foundational framework for the developments that follow.

\begin{definition} \cite{EArtin1926, EArtin1947}
The braid group on $n$ strands, denoted by $B_n$, is the group generated by the elements $
\sigma_1, \sigma_2, \ldots, \sigma_{n-1}$
subject to the relations
\begin{align}
\sigma_i \sigma_{i+1} \sigma_i 
&= \sigma_{i+1} \sigma_i \sigma_{i+1} 
& 1\leq i \leq n-2, \label{eqs1} \\
\sigma_i \sigma_j 
&= \sigma_j \sigma_i 
& |i-j| \geq 2. \label{eqs2}
\end{align}
\noindent The elements $\sigma_i$ are referred to as the braid generators.
\end{definition}

\begin{definition} \cite{baezz, birman}
The singular braid monoid on $n$ strands, denoted by $SM_n$, is the monoid generated by the nonsingular elements $\sigma_1, \sigma_2, \ldots, \sigma_{n-1}$ and the singular elements $\tau_1, \tau_2, \ldots, \tau_{n-1}$. In addition to the relations \eqref{eqs1}--\eqref{eqs2}, the generators of $SM_n$ satisfy the relations
\begin{align}
\tau_i\tau_j &= \tau_j\tau_i 
& |i-j|\geq 2, \label{eqs3}\\
\tau_i\sigma_j &= \sigma_j\tau_i 
& |i-j|\geq 2, \label{eqs4}\\
\tau_i\sigma_i &= \sigma_i\tau_i 
& 1\leq i \leq n-1, \label{eqs5}\\
\sigma_i\sigma_{i+1}\tau_i &= \tau_{i+1}\sigma_i\sigma_{i+1} 
& 1\leq i \leq n-2, \label{eqs6}\\
\sigma_{i+1}\sigma_i\tau_{i+1} &= \tau_i\sigma_{i+1}\sigma_i 
& 1\leq i \leq n-2. \label{eqs7}
\end{align}
\noindent The elements $\tau_i$ are referred to as the singular generators.
\end{definition}
 \noindent If, in addition, the generators $\tau_i$ are assumed to be invertible (i.e., they are nonsingular) for all $1 \leq i \leq n-1$, then one obtains a group that extends the braid group $B_n$. This group is called the singular braid group and is denoted by $SB_n$~\cite{Fenn}.

\begin{definition} \cite{Caprau, Caprau1, Caprau2}
The virtual singular braid monoid on $n$ strands, denoted by $VSM_n$, is the monoid generated by the nonsingular elements $\sigma_1, \sigma_2, \ldots, \sigma_{n-1}$, the singular elements $\tau_1, \tau_2, \ldots, \tau_{n-1}$, and the nonsingular elements $\nu_1, \nu_2, \ldots, \nu_{n-1}$. In addition to the relations \eqref{eqs1}--\eqref{eqs7}, the generators of $VSM_n$ satisfy the relations
\begin{align}
\nu_i^2 &= 1 
& 1\leq i \leq n-1, \label{eqs8}\\
\nu_i\nu_{i+1}\nu_i &= \nu_{i+1}\nu_i\nu_{i+1}
& 1\leq i \leq n-2, \label{eqs9}\\
\nu_i\sigma_{i+1}\nu_i &= \nu_{i+1}\sigma_i\nu_{i+1} 
& 1\leq i \leq n-2, \label{eqs10}\\
\nu_i\tau_{i+1}\nu_i &= \nu_{i+1}\tau_i\nu_{i+1}
& 1\leq i \leq n-2, \label{eqs11}\\
\nu_i\nu_j &= \nu_j\nu_i 
& |i-j|\geq 2, \label{eqs12}\\
\nu_i\sigma_j &= \sigma_j\nu_i 
& |i-j|\geq 2, \label{eqs13}\\
\nu_i\tau_j &= \tau_j\nu_i 
& |i-j|\geq 2. \label{eqs14}
\end{align}
\noindent The elements $\nu_i$ are referred to as the virtual generators. 
\end{definition}

\noindent Similarly, if the generators $\tau_i$ are additionally assumed to be invertible (nonsingular) for all $1 \leq i \leq n-1$, then the resulting group extends both the braid group $B_n$ and the singular braid group $SB_n$. This group is called the virtual singular braid group and is denoted by $VSB_n$ \cite{Caprau2}.

\subsection{The Triplet Group and Its Singular Extensions}

We now introduce the triplet group, together with two related new algebraic structures: the singular triplet monoid and its virtual extension, the virtual singular triplet monoid. These two monoids are analogous, in their algebraic constructions, to the singular braid monoid and the virtual singular braid monoid.

\begin{definition} \cite{MKhovanov1996,Khovanov1997}
The triplet group on $n$ strands, denoted by $L_n$, is the group generated by the elements $
\ell_1, \ell_2, \ldots, \ell_{n-1}$
subject to the relations
\begin{align}
\ell_i^2 &= 1 
& 1\leq i \leq n-1, \label{eqs15} \\
\ell_i \ell_{i+1} \ell_i &= \ell_{i+1} \ell_i \ell_{i+1} 
& 1\leq i \leq n-2. \label{eqs16}
\end{align}
\noindent The elements $\ell_i$ are referred to as the triplet generators.
\end{definition}

\begin{definition} 
The singular triplet monoid on $n$ strands, denoted by $SLM_n$, is the monoid generated by the nonsingular elements $\ell_1, \ell_2, \ldots, \ell_{n-1}$ and the singular elements $\tau_1, \tau_2, \ldots, \tau_{n-1}$. In addition to the relations \eqref{eqs15}--\eqref{eqs16}, the generators of $SLM_n$ are subject to the relations 
\begin{align}
\tau_i\tau_j &= \tau_j\tau_i 
& |i-j|\geq 2, \label{eqs17}\\
\tau_i\ell_j &= \ell_j\tau_i 
& |i-j|\geq 2, \label{eqs18}\\
\tau_i\ell_i &= \ell_i\tau_i 
& 1\leq i \leq n-1, \label{eqs19}\\
\ell_i\ell_{i+1}\tau_i &= \tau_{i+1}\ell_i\ell_{i+1} 
& 1\leq i \leq n-2, \label{eqs20}\\
\ell_{i+1}\ell_i\tau_{i+1} &= \tau_i\ell_{i+1}\ell_i 
& 1\leq i \leq n-2. \label{eqs21}
\end{align}
\end{definition}

 \noindent If, in addition, the generators $\tau_i$ are assumed to be invertible (nonsingular)  for all $1 \leq i \leq n-1$, then the resulting group extends the triplet group $L_n$. We call this group the \textit{singular triplet group} and we denote it by $SL_n$.

\begin{remark} \label{RK6}
The two relations \eqref{eqs20} and \eqref{eqs21} of $SLM_n$ are equivalent, since $\ell_i^2 = 1$ for all $1 \leq i \leq n-1$. Moreover, these relations are equivalent to the following relation:
\[ \ell_i \tau_{i+1} \ell_i = \ell_{i+1} \tau_i \ell_i, \quad 1\leq i \leq n-2.\]
\end{remark}

\begin{definition} \label{monoid-original_presentation}
The virtual singular triplet monoid on $n$ strands, denoted by $VSLM_n$, is the monoid generated by the nonsingular elements $\ell_1, \ell_2, \ldots, \ell_{n-1}$, the singular elements $\tau_1, \tau_2, \ldots, \tau_{n-1}$, and the nonsingular virtual elements $\nu_1, \nu_2, \ldots, \nu_{n-1}$. In addition to the relations \eqref{eqs15}--\eqref{eqs21}, the generators of $VSLM_n$ are subject to the relations
\begin{align}
\nu_i^2 &= 1 
& 1\leq i \leq n-1, \label{eqs22}\\
\nu_i\nu_{i+1}\nu_i &= \nu_{i+1}\nu_i\nu_{i+1} 
& 1\leq i \leq n-2, \label{eqs23}\\
\nu_i\ell_{i+1}\nu_i &= \nu_{i+1}\ell_i\nu_{i+1} 
& 1\leq i \leq n-2, \label{eqs24}\\
\nu_i\tau_{i+1}\nu_i &= \nu_{i+1}\tau_i\nu_{i+1}
& 1\leq i \leq n-2, \label{eqs25}\\
\nu_i\nu_j &= \nu_j\nu_i 
& |i-j|\geq 2, \label{eqs26}\\
\nu_i\ell_j &= \ell_j\nu_i 
& |i-j|\geq 2, \label{eqs27}\\
\nu_i\tau_j &= \tau_j\nu_i 
& |i-j|\geq 2. \label{eqs28}
\end{align}
\end{definition}
 \noindent Again, if we have, in addition, that the generators $\tau_i$ are invertible (nonsingular) for all $1 \leq i \leq n-1$, we then obtain a group that extends both the triplet group $L_n$ and the singular triplet group $SL_n$. We call this group the \textit{virtual singular triplet group} and we denote it by $VSL_n$.

\subsection{Alternative Presentations for $VSLM_n$.}

There is a reduced presentation for the virtual singular triplet monoid using fewer generators, namely $$\{\ell_1, \tau_1,  \nu_1, \nu_2, \dots, \nu_{n-1} \}.$$ For this, we introduce the following defining relations for $\ell_{i+1}$ and $\tau_{i+1}$, for all $1\leq i \leq n-2$:
\begin{eqnarray}
\ell_{i+1}  :=  (\nu_i \ldots \nu_2 \nu_1)(\nu_{i+1}\ldots \nu_3 \nu_2)\ell_1 (\nu_2 \nu_3 \ldots \nu_{i+1})(\nu_1 \nu_2 \ldots \nu_i),  \label{def_rel_ell} \\ 
\tau_{i+1}  :=  (\nu_i \ldots \nu_2 \nu_1)(\nu_{i+1}\ldots \nu_3 \nu_2)\tau_1 (\nu_2 \nu_3 \ldots \nu_{i+1})(\nu_1 \nu_2 \ldots \nu_i). \label{def_rel_tau}  
\end{eqnarray}
For example, $\ell_2 := \nu_1\nu_2 \ell_1 \nu_2 \nu_1$ and $\ell_3 :=  (\nu_2 \nu_1)(\nu_3 \nu_2) \ell_1 (\nu_2 \nu_3) (\nu_1 \nu_2)$. Similarly,
$\tau_2 := \nu_1\nu_2 \tau_1 \nu_2 \nu_1$ and $\tau_3 :=  (\nu_2 \nu_1)(\nu_3 \nu_2) \tau_1 (\nu_2 \nu_3) (\nu_1 \nu_2)$.

From the original presentation, we keep all of the relations involving only the virtual generators $\nu_i$. For the relations involving the singular and nonsingular elements, it suffices to impose only the relations that occur on the left of a braid (e.q. relations describing interactions among the first strands of a braid). The relations \eqref{eqs24} and \eqref{eqs25} are not needed in the reduced presentation since they were used in the defining relations \eqref{def_rel_ell} and  \eqref{def_rel_tau}.

\begin{theorem} \label{thm:reduced}
The virtual singular triplet monoid $VSLM_n$ has the following reduced presentation with generators $\{\ell_1, \tau_1, \nu_1, \nu_2, \dots, \nu_{n-1} \}$
and relations:
\begin{enumerate}
\item [(8.1)] $\nu_i^2 = 1$ for all $1\leq i \leq n-1$,
\item [(8.2)] $\nu_i \nu_j \nu_i = \nu_j \nu_i \nu_j$, where $|i-j| = 1$,
\item [(8.3)] $\nu_i \nu_j = \nu_j \nu_i$, where $|i-j| \geq 2$,
\item [(8.4)] $\ell_1^2 = 1$,
\item [(8.5)] $\ell_1\tau_1 =  \tau_1 \ell_1$,
\item [(8.6)] $\ell_1 ( \nu_1 \nu_2 \ell_1 \nu_2 \nu_1 ) \ell_1 = (\nu_1 \nu_2 \ell_1 \nu_2 \nu_1 ) \ell_1(\nu_1 \nu_2 \ell_1 \nu_2 \nu_1 ),$  
\item [(8.7)] $\ell_1 ( \nu_1 \nu_2 \tau_1 \nu_2 \nu_1 ) \ell_1 = (\nu_1 \nu_2 \tau_1 \nu_2 \nu_1 ) \ell_1(\nu_1 \nu_2 \tau_1 \nu_2 \nu_1 ),$ 
\item [(8.8)] $\ell_1 \nu_i = \nu_i \ell_1$ and $\tau_1 \nu_i = \nu_i \tau_1$ for all $i \geq 3$, 
\item [(8.9)] $\tau_1 (\nu_2  \nu_1 \nu_3 \nu_2 \tau_1 \nu_2 \nu_3 \nu_1 \nu_2) = (\nu_2 \nu_1 \nu_3 \nu_2 \tau_1 \nu_2 \nu_3 \nu_1 \nu_2) \tau_1$, 
\item [(8.10)] $\tau_1  (\nu_2  \nu_1 \nu_3 \nu_2 \ell_1 \nu_2 \nu_3 \nu_1 \nu_2)   =  (\nu_2  \nu_1 \nu_3 \nu_2 \ell_1 \nu_2 \nu_3 \nu_1 \nu_2)  \tau_1$. 
\end{enumerate}
\end{theorem}

\begin{proof}
We need to prove that every defining relation of the original presentation of $VSLM
_n$ follows from the reduced presentation.

(a) We begin by observing that the defining relations~\eqref{def_rel_ell} immediately imply that $\ell_i^2 = 1$ for all $i\geq 2$. Therefore, when combining with the base case relation (8.4), it follows that relations~\eqref{eqs15} hold for all $1\leq i \leq n-1$. 

(b) The proof that relations~\eqref{eqs17} and~\eqref{eqs18} follow from the reduced presentation is analogous to the proof of~\cite[Lemma 27, part (ii)]{CaprauNasser}, after replacing $s_i$ with $\ell_i$.

(c) We next  show that relations~\eqref{eqs19} hold. We start by employing the defining relations~\eqref{def_rel_ell} and \eqref{def_rel_tau} for $\tau_i$ and $\ell_i$, respectively.
\begin{align*}
\tau_i \ell_i =& [ (\nu_{i-1}\ldots \nu_2 \nu_1)(\nu_{i}\ldots \nu_3 \nu_2)\tau_1  (\nu_2 \nu_3 \ldots \nu_{i}) (\nu_1\nu_2 \ldots \nu_{i-1}) ]\,  [ (\nu_{i-1}\ldots \nu_2\nu_1) \cdot \\ 
& (\nu_{i}\ldots \nu_3 \nu_2)\ell_1 (\nu_2 \nu_3 \ldots \nu_{i})(\nu_1 \nu_2\ldots \nu_{i-1}) ]\\
\stackrel{(8.3)}{=}&  (\nu_{i-1}\ldots \nu_2 \nu_1)(\nu_i \ldots \nu_3 \nu_2)\tau_1  \ell_1  (\nu_2 \nu_3 \ldots \nu_i)(\nu_1 \nu_2 \ldots \nu_{i-1}). 
\end{align*}
Similarly, we obtain:
\[
 \ell_i \tau_i=(\nu_{i-1}\ldots \nu_2 \nu_1)(\nu_i \ldots \nu_3 \nu_2) \ell_1  \tau_1(\nu_2 \nu_3 \ldots \nu_i)(\nu_1 \nu_2 \ldots \nu_{i-1}). 
\]
Since $\tau_1 \ell_1 = \ell_1 \tau_1$, it follows that $\tau_i \ell_i =  \ell_i \tau_i$, for all $1\leq i \leq n-1$.

(d) Relations~\eqref{eqs20} and \eqref{eqs21} are equivalent, and are established using similar steps as in~\cite[Lemma 28]{CaprauNasser}, with $s_i$ replaced by $\ell_i$.

(e) We next prove that relations~\eqref{eqs24} hold.  For $i = 1$, the relation $\nu_1 \ell_2 \nu_1 = \nu_2 \ell_1 \nu_2$ follows immediately from the defining relation $\ell_2 = \nu_1 \nu_2 \ell_1 \nu_2 \nu_1$ and the relation $\nu_1^2 = 1$. Now assume that $i \geq 2$. Then, 
\begin{align*}
\nu_i \ell_{i+1} \nu_i \stackrel{\eqref{def_rel_ell}}{=}& \underline{\nu_i  (\nu_i} \ldots \nu_2 \nu_1)(\nu_{i+1}\ldots \nu_3 \nu_2)\ell_1 (\nu_2 \nu_3 \ldots \nu_{i+1})(\nu_1\nu_2\ldots \underline{\nu_i) \nu_i}\\
\stackrel{(8.1)}{=}&  (\nu_{i-1}\ldots \nu_2 \nu_1)(\underline{\nu_{i+1}}\ldots \nu_3 \nu_2) \ell_1 (\nu_2 \nu_3\ldots \underline{\nu_{i+1}})(\nu_1\nu_2\ldots \nu_{i-1}) \\
\stackrel{(8.3)}{=}& \nu_{i+1}  (\underline{\nu_{i-1}\ldots \nu_2 \nu_1)(\nu_i\ldots \nu_3 \nu_2)\ell_1 (\nu_2 \nu_3 \ldots \nu_i)(\nu_1 \nu_2 \ldots \nu_{i-1}}) \nu_{i+1}\\
\stackrel{\eqref{def_rel_ell}}{=}& \nu_{i+1} \ell_i \nu_{i+1}.
\end{align*}
At each step above, we underlined the subword to which a defining relation is applied and indicate above the equality sign the relation being used. 
Therefore, $\nu_i \ell_{i+1} \nu_i = \nu_{i+1} \ell_i \nu_{i+1}$ for all $1\leq i \leq n-2$.  The proof of relation~\eqref{eqs25} is identical, replacing $\ell_i$ with $\tau_i$ throughout.

(f) We prove next that relations~\eqref{eqs27} and~\eqref{eqs28} follow from the reduced presentation.
By the defining relations~\eqref{def_rel_ell} for $\ell_i$, we have:
  \[\ell_i \nu_j =(\nu_{i-1}\ldots
\nu_2 \nu_1)\,(\nu_i \ldots \nu_3 \nu_2)\,\ell_1\, (\nu_2 \nu_3 \ldots \nu_i)\,(\nu_1 \nu_2\ldots \nu_{i-1})\, \nu_j.\]  
If $|i-j| \geq 2$, then either $j\geq i+2$ or $j\leq i-2$. If $j\geq i+2$, then $\nu_j$  commutes  with all generators in the above
expression, thus $\ell_i \nu_j  =  \nu_j \ell_i$ in this case. 
 If $j\leq i-2$, we apply the following steps, where we move $\nu_j$ successively to the left using the braid relations (8.2) and the commutation relations (8.3):
\begin{align*}
\ell_i \nu_j  \stackrel{\eqref{def_rel_ell}}{=}&  (\nu_{i-1}\ldots \nu_2\nu_1)(\nu_i\ldots \nu_3\nu_2)\ell_1 (\nu_2\nu_3\ldots \nu_i)(\nu_1\nu_2\ldots \nu_{i-1})\underline{\nu_j}  \\
\stackrel{(8.3)}{=}&  (\nu_{i-1}\ldots \nu_1)(\nu_i\ldots \nu_2)\ell_1 (\nu_2\nu_3\ldots \nu_i)(\nu_1\nu_2 \ldots \nu_{j-1}\underline{\nu_j \nu_{j+1} \nu_j}  \nu_{j+2} \ldots \nu_{i-1})  \\ 
\stackrel{(8.2)}{=}&  (\nu_{i-1}\ldots \nu_1)(\nu_i\ldots \nu_2)\ell_1 (\nu_2 \nu_3\ldots  \nu_i)(\nu_1\nu_2 \ldots \nu_{j-1} \underline{\nu_{j+1}} \nu_j  \nu_{j+1} \nu_{j+2} \ldots \nu_{i-1}) \\
\stackrel{(8.3)}{=}&  (\nu_{i-1}\ldots \nu_1)(\nu_i\ldots \nu_3\nu_2)\ell_1 (\nu_2\nu_3\ldots  \nu_j \underline{ \nu_{j+1} \nu_{j+2} \nu_{j+1}}  \nu_{j+3} \ldots \nu_i) (\nu_1\nu_2 \ldots \nu_{i-1}) \\
\stackrel{(8.2)}{=}&  (\nu_{i-1}\ldots \nu_1)(\nu_i\ldots \nu_3 \nu_2)\ell_1(\nu_2 \nu_3 \ldots  \nu_j\underline{ \nu_{j+2}} \nu_{j+1} \nu_{j+2} \nu_{j+3} \ldots \nu_i)  (\nu_1 \nu_2 \ldots \nu_{i-1}) \\
 \displaystyle  \mathop{=}_{(8.8)}^{(8.3)}&  (\nu_{i-1}\ldots \nu_1)(\nu_i\ldots \nu_{j+3} \underline{ \nu_{j+2}  \nu_{j+1} \nu_{j+2}} \nu_j\ldots \nu_2) \ell_1 (\nu_2\ldots  \nu_i) (\nu_1 \ldots \nu_{i-1}).
 \end{align*}
  Applying again relations~(8.2) and~(8.3) repeatedly, we obtain:
 \begin{align*}
 \ell_i \nu_j \stackrel{(8.2)}{=}&  (\nu_{i-1}\ldots \nu_2 \nu_1)(\nu_i \ldots  \nu_{j+3} \underline{ \nu_{j+1}} \nu_{j+2} \nu_{j+1} \nu_j\ldots \nu_2)\ell_1 (\nu_2 \ldots  \nu_i)  (\nu_1 \ldots \nu_{i-1}) \\
\stackrel{(8.3)}{=}&  (\nu_{i-1}\ldots \nu_{j+2} \underline{\nu_{j+1} \nu_j \nu_{j+1} } \nu_{j-1} \ldots \nu_2 \nu_1)  (\nu_i\ldots \nu_3 \nu_2)\ell_1 (\nu_2 \ldots  \nu_i) (\nu_1 \ldots \nu_{i-1}) \\
\stackrel{(8.2)}{=}& (\nu_{i-1}\ldots \nu_{j+2} \underline{\nu_j }  \nu_{j+1}  \nu_j  \nu_{j-1} \ldots \nu_2 \nu_1)  (\nu_i\ldots \nu_3\nu_2) \ell_1(\nu_2\ldots  \nu_i) (\nu_1 \ldots \nu_{i-1}) \\
\stackrel{(8.3)}{=}&  \nu_j (\nu_{i-1}\ldots  \nu_1)  (\nu_i\ldots \nu_2) \ell_1 (\nu_2\ldots  \nu_i) (\nu_1 \ldots \nu_{i-1}) 
 \displaystyle  \stackrel{\eqref{def_rel_ell}}{=}   \nu_j \ell_i.  
\end{align*} 
Therefore, $\ell_i \nu_j = \nu_j\ell_i$ for all $|i-j| \geq 2$. The proof of the relation $\tau_i \nu_j = \nu_j \tau_i$, where $|i-j| \geq 2$, is identical, replacing $\ell_i$ with $\tau_i$, and is therefore omitted.

(g) It remains to prove that relations~\eqref{eqs16} hold. The strategy is to expand both sides using the defining expressions for $\ell_i$, simplify the right-hand side until it matches the expansion of the left-hand side, and finally invoke the base relation~(8.6). We begin by observing that the following identity holds among the virtual generators (see \cite[Lemma 24]{CaprauNasser}):
\begin{align} \label{eqs_virtuals_prep} 
\nu_i \nu_{i-1}\ldots { \nu_{j+1}{\nu_j} \nu_{j+1}}\ldots \nu_{i-1} \nu_i =  \nu_j \nu_{j+1}\ldots \nu_{i-1}{\nu_i} \nu_{i-1}\ldots   \nu_{j+1} \nu_j.
 \end{align}  
The base case of the relation ~\eqref{eqs16} corresponds to $i = 1$ and $j = 2$, and is precisely  relation~(8.6) of the reduced presentation. We therefore assume  that $i\geq 2$.\\
We first expand the left-hand side of the desired identity. Using relations~\eqref{eqs21},~\eqref{def_rel_ell}, and \eqref{eqs_virtuals_prep}, we obtain:
\begin{align*}
\ell_i  \ell_{i+1} \ell_i = & (\nu_{i-1}\ldots \nu_1)(\nu_{i}\ldots \nu_2) ( \nu_{i+1} \ldots \nu_3) ( \ell_1   \nu_1 \nu_2  \ell_1 \nu_2 \nu_1  \ell_1)( \nu_3  \ldots  \nu_{i+1} ) \cdot  
\\&  (\nu_2\ldots \nu_i)(\nu_1\ldots \nu_{i-1}).
\end{align*}
Similarly, expanding the right-hand side yields,
\begin{align*}
\ell_{i+1} \underline{ \ell_i} \ell_{i+1} \stackrel{\eqref{def_rel_ell}}{=}& \ell_{i+1} (\nu_{i-1}\ldots \nu_1)  (\nu_i \ldots  \nu_2 \underline{(1_n)}) \ell_1 \underline{(1_n)(1_n)}(\nu_2 \ldots \nu_{i}) (\nu_1\ldots \nu_{i-1})  \ell_{i+1} \\
=\,\,\, & \underline{ \ell_{i+1} } (\nu_{i-1}\ldots \nu_1)  (\nu_i \ldots \nu_2 \nu_1 \nu_1) \ell_1 \underline{(\nu_{i+1}\ldots \nu_3)}   (\nu_3 \ldots  \nu_{i+1})  (\underline{\nu_1} \nu_1 )\cdot 
\\&( \nu_2\ldots \nu_{i}) (\nu_1\ldots \nu_{i-1})\underline{\ell_{i+1}}.
\end{align*}
Applying relations~\eqref{eqs27}, (8.3), and~(8.8) repeatedly, we obtain:
\begin{align*}
\ell_{i+1}  \ell_i \ell_{i+1}
=\,\,\,&  (\nu_{i-1}\ldots \nu_1) \ell_{i+1}  (\nu_i \ldots \nu_{1})(\nu_{i+1}\ldots \nu_3)\underline{(1_n)}  \nu_1  \ell_1 \,\nu_1 \underline{(1_n)} (\nu_3 \ldots  \nu_{i+1} ) \cdot 
\\& (\nu_1\ldots \nu_{i})\ell_{i+1} (\nu_1\ldots \nu_{i-1}) \\
=\,\,\,&  (\nu_{i-1}\ldots \nu_1) \underline{\ell_{i+1}} (\nu_i \ldots \nu_{1})(\nu_{i+1}\ldots \nu_3) ( \nu_2   \nu_2 ) \nu_1  \ell_1 \nu_1 (\nu_2   \nu_2 )( \nu_3 \ldots  \nu_{i+1} ) \cdot 
\\& (\nu_1\ldots \nu_i) \underline{\ell_{i+1}} (\nu_1\ldots \nu_{i-1}) \\
\stackrel{\eqref{def_rel_ell}}{=}& (\nu_{i-1}\ldots \nu_1)  (\nu_{i}\ldots \nu_1) ( \nu_{i+1} \ldots \nu_2 ) \ell_1 \underline{(\nu_2 \ldots \nu_{i+1})(\nu_1 \ldots \nu_i)(\nu_i \ldots \nu_1)(\nu_{i+1}\ldots \nu_2) } \cdot
\\& \nu_2 \nu_1  \ell_1 \nu_1 \nu_2  \underline{( \nu_2 \dots  \nu_{i+1} ) (\nu_1\dots \nu_i) (\nu_i \dots \nu_1) ( \nu_{i+1} \dots \nu_2 )}\ell_1 ( \nu_2 \dots  \nu_{i+1} ) (\nu_1\dots \nu_i) \cdot
\\&  (\nu_1\dots \nu_{i-1})\\
\stackrel{(8.1)}{=}&  (\nu_{i-1}\ldots \nu_1)(\nu_i \ldots \underline{\nu_1}) ( \nu_{i+1} \ldots \nu_2 ) \ell_1 \nu_2 \nu_1  \ell_1 \nu_1 \nu_2 \ell_1 ( \nu_2 \ldots  \nu_{i+1} )     (\underline{\nu_1}\ldots \nu_i) \cdot
\\& (\nu_1\ldots \nu_{i-1}).
\end{align*}
Finally, using the commutation relations~(8.3), we move the underlined left occurrence of $\nu_1$ as far to the right as possible and the underlined right occurrence as far to the left. This yields:
\begin{align*}
\ell_{i+1}  \ell_i \ell_{i+1}
\stackrel{(8.3)}{=}&  (\nu_{i-1}\ldots \nu_1)(\nu_i\ldots \nu_2) ( \nu_{i+1} \ldots \nu_3)(\underline{\nu_1 \nu_2 \ell_1  \nu_2 \nu_1 ) \ell_1(\nu_1 \nu_2 \ell_1 \nu_2  \nu_1 }  )( \nu_3 \ldots  \nu_{i+1} )    \cdot 
\\&   (\nu_2\ldots \nu_i) (\nu_1\ldots \nu_{i-1}) \\
\stackrel{(8.6)}{=}&  (\nu_{i-1}\ldots \nu_1)(\nu_i \ldots \nu_2) ( \nu_{i+1} \ldots \nu_3) ( \ell_1   \nu_1 \nu_2  \ell_1 \nu_2 \nu_1  \ell_1  )( \nu_3  \ldots  \nu_{i+1} ) (\nu_2 \ldots \nu_i) \cdot 
\\& (\nu_1\dots \nu_{i-1}). 
\end{align*}
The final expression coincides with the expansion of $\ell_i  \ell_{i+1} \ell_i$. Hence,
$ \ell_i  \ell_{i+1} \ell_i = \ell_{i+1}  \ell_i \ell_{i+1}$ 
for all $i \geq 2$. Together with the base case, this establishes relation~\eqref{eqs16}.

Therefore, every defining relation of the original presentation of $VSLM_n$ follows from the reduced presentation.
\end{proof}

We note that there is no analogous reduced presentation for the singular triplet monoid $SLM_n$, since the existence of a reduced presentation for $VSLM_n$ relies on the presence of the virtual generators and the relations among them.

\section{The Virtual Singular Triplet Monoid Via Triplet Connecting Strings} \label{sec:newpres}

There exists an obvious homomorphism $\pi$ from $VSLM_n$ to $S_n$, defined on its generators as follows:
\[ \pi: VSLM_n \longrightarrow S_n, \,\, \pi(\ell_i) =\pi(\tau_i) =\pi(\nu_i) =(i \ \   i+1), \,\,\, \text{for all} \,\, \, 1 \leq i \leq n-1. \]
The elements of $VSLM_n$ that induce the identity permutation deserve special attention. We refer to the kernel of the homomorphism $\pi$ as the virtual singular pure triplet monoid on $n$ strands, and we denote it by $VSPLM_n$. 

The goal of this section is to introduce a new presentation of the virtual singular triplet monoid $VSLM_n$ in terms of certain distinguished elements of the normal submonoid $VSPLM_n$, which we now define.   
 
 \subsection{Triplet Connecting Strings and Their Relations} \label{sec:ConnectingStrings}
 
  \begin{definition}\label{def:cstrings}
  The \textit{triplet connecting strings} $\mu_i$ and $\gamma_i$, where $1 \leq i \leq n-1$, are the elements defined as follows:
   \[\mu_i: = \ell_i \nu_i \  \text{ and } \  \gamma_i: = \tau_i \nu_i.\]
    \end{definition} 
It is immediate from the defining relations that the elements $\mu_i$ are invertible with $\mu_i^{-1} = \nu_i \ell_i$. On the other hand, the elements $\gamma_i$ are not invertible.

\begin{lemma}\label{lemma:rel_connecting_strings}
The following relations hold in $VSLM_n$.
\begin{enumerate}
\item [(i)]  $\nu_i\mu_{j} \nu_{i} = \nu_{j} \mu_{i} \nu_{j} \ \text{ and }\  \nu_i\gamma_{j} \nu_{i} = \nu_{j}\gamma_{i} \nu_{j}$ for $|i-j|=1$
\item [(ii)] $\mu_{j}( \nu_{j}\mu_{i} \nu_{j})\mu_i=\mu_{i}( \nu_{j}\mu_{i} \nu_{j})\mu_{j}$ for $|i-j|=1$
\item [(iii)] $\mu_{j}( \nu_{j}\mu_{i} \nu_{j})\gamma_i=\gamma_{i}( \nu_{j}\mu_{i} \nu_{j})\mu_{j}$ for $|i-j|=1$
 \item [(iv)] $\mu_{i} \nu_i\gamma_i = \gamma_i \nu_i\mu_{i}$ for all $1\leq i \leq n-1$
 \item [(v)] $\gamma_i \gamma_j = \gamma_j \gamma_i, \ \mu_i\gamma_j=\gamma_j\mu_i, \ \nu_i \mu_j = \mu_j \nu_i, \text{ and }\ \nu_i \gamma_j = \gamma_j \nu_i$, where $|i-j| \geq 2$.
\end{enumerate}
\end{lemma}

\begin{proof}
The proofs of identities (i), (iii), and (iv) are analogous to those of the corresponding identities in \cite[Lemma 30]{CaprauNasser}, with $s_i$ replaced by $\ell_i$. We prove identity (ii) below.
At each step, we underline the subword to which a defining relation of $VSLM_n$ is applied and indicate above the equality sign the relation being used. Throughout the proof, we use the following equivalent form of relations~\eqref{eqs24}:
\[ \nu_i \nu_j \ell_i = \ell_j \nu_i \nu_j, \,\, \text{where} \,\, |i-j| = 1.\]
In the first step below, we make use of Definition~\ref{def:cstrings}. 
\begin{align*}
\mu_{j}(\nu_{j}\mu_{i} \nu_{j})\mu_i &= \ell_{j}\underline{\nu_{j} \nu_{j}} \ell_i \nu_i \nu_{j}\ell_i \nu_i \stackrel{\eqref{eqs22}}{=} \ell_{j}\ell_i\underline{\nu_i \nu_{j} \ell_i} \nu_i \stackrel{\eqref{eqs24}}{=}\underline{\ell_{j}\ell_i\ell_{j}}\,\,\,\underline{\nu_i \nu_{j} \nu_i} \stackrel{\eqref{eqs16}, \eqref{eqs23}}{=}\ell_{i}\ell_{j}\underline{\ell_{i}\nu _{j} \nu_{i}}\nu_{j}\\
&\stackrel{\eqref{eqs24}}{=} \ell_{i}\ell_{j}\nu_{j}\nu_{i}\ell_{j}\nu_{j} = \ell_{i}\ell_{j}\nu_{j}\nu_{i}\underline{1_n}\ell_{j}\nu_{j} \stackrel{\eqref{eqs22}}{=}  \ell_{i}\ell_{j}\underline{\nu_{j}\nu_{i}\nu_{j}}\nu_{j}\ell_{j}\nu_{j}.
\end{align*}
Applying identities~\eqref{eqs23} and~\eqref{eqs24} once more, together with the definition of the triplet connecting strings $\mu_i$, yields the desired result:
\begin{align*}
\mu_{j}( \nu_{j}\mu_{i}\nu_{j})\mu_i &=\ell_{i}\ell_{j}\underline{\nu_{j}\nu_{i}\nu_{j}}\nu_{j}\ell_{j}\nu_{j} \stackrel{\eqref{eqs23}}{=} \ell_{i}\underline{\ell_{j}\nu_{i}\nu_{j} }\nu_{i}\nu_{j}\ell_{j}\nu_{j} \stackrel{\eqref{eqs24}}{=}  
\underline{\ell_{i}\nu_{i}}\nu_{j}\underline{\ell_{i}\nu_{i}}\nu_{j}\underline{\ell_{j}\nu_{j}}\\
&= \mu_i(\nu_{j}\mu_i\nu_{j})\mu_{j}.
\end{align*}
Relations (v) are immediate consequences of identities~\eqref{eqs17}, \eqref{eqs18}, \eqref{eqs27}, and~\eqref{eqs28}.
\end{proof}

\subsection{Presentations for $VSLM_n$ Using the Triplet Connecting Strings} \label{sec:PresFusing}

In this subsection, we introduce a presentation of the monoid $VSLM_n$ in terms of the triplet connecting strings. To this end, we first define a monoid whose generators are the triplet connecting strings together with the virtual generators. The relations established in Lemma~\ref{lemma:rel_connecting_strings} are taken as defining relations of this monoid.

\begin{definition} \label{def:Monoid M_n}
Let $M_n$ be the monoid with the following presentation using generators $\{\mu_i^{\pm 1},  \gamma_i, \nu_i  \, \big{\vert}  \,1\leq i \leq {n-1}\}$ and relations:
     \begin{align*}
     \nu_i^2=1 \,\, &\text{and}\,\, \, \mu_i\mu_i^{-1}=1=\mu_i^{-1}\mu_i \,\,\, \text{for all } 1\le i\le n-1,  \\
       \nu_i \nu_j \nu_i&= \nu_j \nu_i \nu_j,  \,\,  |i-j|=1, \\
       \nu_i \mu_j \nu_i&= \nu_j \mu_i \nu_j,  \,\,  |i-j|=1, \\
 \nu_i \gamma_j \nu_i&= \nu_j \gamma_i \nu_j, \,\,  |i-j|=1, \\
 \mu_{j}( \nu_{j}\mu_{i} \nu_{j})\mu_i &=\mu_{i}( \nu_{j}\mu_{i} \nu_{j})\mu_{j},  \,\, |i-j|=1,  \\
       \mu_j(\nu_j \mu_i \nu_j) \gamma_i &=\gamma_i (\nu_j \mu_i \nu_j) \mu_j, \,\,   |i-j|=1, \\
         \mu_i \nu_i\gamma_i &=\gamma_i \nu_i \mu_i,  \,\,\, \text{for all } 1\le i\le n-1,  \\
  \gamma_i \gamma_j &= \gamma_j \gamma_i \,\, \text{and} \,\,  \mu_i\gamma_j=\gamma_j\mu_i,  \,\, |i-j| \geq 2,  \\
   \nu_i \mu_j &= \mu_j \nu_i, \nu_i \gamma_j = \gamma_j \nu_i \,\, \text{and} \,\,  \nu_i \nu_j = \nu_j \nu_i, \,\, |i-j| \geq 2.  
\end{align*} 
\end{definition}

\begin{theorem}\label{theorem:iso}
The monoids $M_n$ and $VSLM_n$ are isomorphic.
\end{theorem}   

\begin{proof}
Consider the map $f: M_n \longrightarrow VSLM_n$ defined on generators by 
\[f(\nu_i) = \nu_i, \quad f(\mu_i) = \ell_i \nu_i,   \quad  f(\gamma_{i})=\tau_i \nu_i,\]
and extend it multiplicatively to all elements of $M_n$. By Lemma~\ref{lemma:rel_connecting_strings}, together with the definition for the monoid $M_n$, the map $f$ preserves the relations for $M_n$. Hence $f$ is a well-defined monoid homomorphism. Next, define the homomorphism $g: VSLM_n\longrightarrow M_n$ defined on the generators of $VSLM_n$ by, 
 \[g(\nu_i) = \nu_i,\quad g(\ell_i) = \mu_i \nu_i, \quad g(\tau_i) = \gamma_i \nu_i.\]
It is straightforward to verify that the map $g$ preserves the defining relations of $VSLM_n$, and thus $g$ is a well-defined monoid homomorphism. Moreover, $f\circ g = \id_{VSLM_n}$ and $g\circ f = \id_{M_n}$. Therefore, $f$ and $g$ are inverse isomorphisms, and consequently, $M_n$ and $VSLM_n$ are isomorphic.
\end{proof}

\begin{remark}
Since $VSLM_n$ is isomorphic to $M_n$, the presentation of $M_n$ given in Definition~\ref{def:Monoid M_n} yields an alternative presentation for the virtual singular triplet monoid $VSLM_n$ with generators $\mu_i^{\pm 1}, \gamma_i$ and $\nu_i$.\\
There is a reduced presentation for the monoid $M_n$ by expressing the generators $\mu_i^{\pm 1}$ and $\gamma_i$ in terms of 
 $\mu_1^{\pm 1}$ and $\gamma_1$, respectively, as follows:
\[\mu_{i+1}^{\pm{1}} = (\nu_i\dots \nu_1)(\nu_{i+1}\dots \nu_3 \nu_2)\mu_1^{\pm{1}}(\nu_2 \nu_3 \dots \nu_{i+1})(\nu_1\nu_2\dots \nu_i),\]
\[\gamma_{i+1} = (\nu_i\dots \nu_1)(\nu_{i+1}\dots \nu_3 \nu_2)\gamma_1(\nu_2 \nu_3\dots \nu_{i+1})(\nu_1\nu_2\dots \nu_i).\]
Consequently, every relation in Definition~\ref{def:Monoid M_n} can be rewritten using only the generators $\mu_1^{\pm{1}}, \gamma_1$, and $\nu_1, \dots, \nu_{i-1}$. It follows that $VSLM_n$ admits a reduced presentation with generators $\mu_1^{\pm{1}},\gamma_1$, and $\nu_i$, for all $1\leq i \leq n-1$, as described by the following statement.
\end{remark}

Substituting the expressions for $\mu_i$ and $\gamma_i$ from the preceding remark into the relations of Definition~\ref{def:Monoid M_n}, one obtains the following reduced presentation.

\begin{proposition}\label{prop:redvssn}
The monoid $VSLM_n$ admits the following reduced presentation. It is generated by $\{\nu_1, \nu_2,\dots , \nu_{n-1}\} \cup \{\mu_1, \mu_1^{-1}, \gamma_1\}$, subject to the relations:
\begin{align*}
     \nu_i^2=1 \,\, &\text{and}\,\, \,  \mu_{1}^{-1}\mu_{1}=1 = \mu_{1}\mu^{-1}_{1},\\
          \nu_i \nu_j \nu_i&= \nu_j \nu_i \nu_j, \,\, |i-j|=1,\\
       (\nu_1 \nu_2 \mu_{1}\nu_2 \nu_1)(\nu_2\mu_{1}\nu_2)\mu_{1}&=\mu_{1}(\nu_2\mu_{1}\nu_2)(\nu_1\nu_2\mu_{1}\nu_2 \nu_1), \\
      (\nu_1\nu_2\mu_{1}\nu_2\nu_1)(\nu_2\mu_{1}\nu_2)\gamma_{1}&=\gamma_{1}(\nu_2\mu_{1}\nu_2)(\nu_1\nu_2\mu_{1}\nu_2\nu_1),\\          \mu_{1}\nu_1\gamma_{1}&=\gamma_{1}\nu_1\mu_{1},\\
       \nu_i\nu_j&=\nu_j\nu_i, \,\, |i-j|>1, \\
   \mu_{1} \nu_i  = \nu_i \mu_1\,\,\, &\text{and}\,\, \,\gamma_1\nu_i=\nu_i\gamma_1, \,\, i \geq 3,\\
        \gamma_{1} (\nu_2 \nu_1 \nu_3 \nu_2 \gamma_{1} \nu_2 \nu_3 \nu_1 \nu_2) &=(\nu_2 \nu_1 \nu_3 \nu_2 \gamma_{1} \nu_2 \nu_3 \nu_1 \nu_2)\gamma_{1},  \\
        \gamma_{1} (\nu_2 \nu_1 \nu_3 \nu_2 \mu_{1} \nu_2  \nu_3 \nu_1 \nu_2)& =(\nu_2 \nu_1 \nu_3 \nu_2 \mu_{1} \nu_2 \nu_3 \nu_1 \nu_2)\gamma_{1}. 
 \end{align*} 
\end{proposition}
\begin{proof}
The proof is analogous to that of Theorem~\ref{thm:reduced}. Since the generators $\mu_i$ and $\gamma_i$ for $i\geq 2$ are defined in terms of $\mu_1$ and $\gamma_1$, and the virtual generators, they need not be included among the generators of the reduced presentation. On the other hand, all relations involving only the virtual generators must be retained.

 The relations $\nu_i \mu_j \nu_i= \nu_j \mu_i \nu_j$ and $\nu_i \gamma_j \nu_i= \nu_j \gamma_i \nu_j$ for $|i-j|=1$ are omitted, since they are used to define the generators $\mu_{i+1}$ and $\gamma_{i+1}$. 
 
 From relations $\mu_i\mu_i^{-1}=1=\mu_i^{-1}\mu_i$ and $\mu_i \nu_i\gamma_i =\gamma_i \nu_i \mu_i$, we only need $\mu_1\mu_1^{-1}=1=\mu_1^{-1}\mu_1$  and $\mu_1 \nu_1 \gamma_1 =\gamma_1\nu_1 \mu_1$, respectively.
  
 For relations involving two indices, we assume that they occur in the leftmost part of the braid and therefore choose the smallest possible values of the indices. In particular, for relations with $|i-j|=1$, we take $i = 1$ and $j=2$, replacing $\mu_2$ and $\gamma_2$ by their expressions in terms of $\mu_1, \gamma_1$, and the virtual generators. Likewise, for relations with $|i-j|\geq2$, we take $i=1$ and $j=3$, replacing $\mu_3$ and $\gamma_3$ accordingly. For example, $\mu_2$ is replaced by $\nu_1\nu_2\mu_{1}\nu_2\nu_1$, while $\mu_3$ is replaced by $\nu_2 \nu_1 \nu_3 \nu_2 \mu_1 \nu_2 \nu_3 \nu_1 \nu_2$. The last four relations in the statement of the proposition account for all commuting relations with $|i-j| \geq 2$.
 
 Finally, any relation occurring elsewhere in the braid follows from these leftmost relations by repeated applications of the detour move—that is, by `sliding' strands across any portion of the braid using only virtual generators—as in the proof of Theorem~\ref{thm:reduced}. This completes the proof.
 \end{proof}

\section{On Extending Representations of $L_n$ to $SLM_n$ and $VSLM_n$} \label{rep-extensions}

In this section, we work with the presentation for the monoid $VSLM_n$ as given in Definition~\ref{monoid-original_presentation}. In order to extend any representation $\lambda$ of $L_n$ to $SLM_n$ and $VSLM_n$, we need to construct maps that extend $\lambda$ to the other families of generators in such a way that the defining relations of $SLM_n$ and $VSLM_n$ are preserved. There may exist several possible types of extension. In this work, we consider two types of extension for certain families of representations of $L_n$ to $SLM_n$ and $VSLM_n$. We first give the concept of $k$-local representations of a group $G$ for an integer $k$.

\begin{definition} \cite{M.N1} \label{deflocal}
Let $G$ be a group with generators $g_1,g_2,\ldots, g_{n-1}$. A matrix representation $\lambda: G \longrightarrow \mathrm{GL}_{m}(\mathbb{Z}[t^{\pm 1}])$, where $t$ is an indeterminate, is called \emph{$k$-local} if the image of each generator $g_i$ has the block form
\[
\lambda(g_i)=
\begin{pmatrix}
I_{i-1} & 0 & 0 \\
0 & \left( M_i\right) & 0 \\
0 & 0 & I_{n-i-1}
\end{pmatrix},
\qquad 1\le i \le n-1,
\]
where $M_i \in \mathrm{GL}_{k}(\mathbb{Z}[t^{\pm 1}])$, $k = m-n+2$, and $I_r$ denotes the $r\times r$ identity matrix. The representation $\lambda$ is said to be homogeneous if all the matrices $M_i$ are identical. 
\end{definition}

\begin{remark}
If $G$ in Definition \ref{deflocal} is assumed to be a monoid, then the $k$-local representation $\lambda$ naturally induces a representation of $G$ into $\mathrm{M}_{m}\!\left(\mathbb{Z}[t^{\pm 1}]\right)$.
\end{remark}
 
\begin{remark}
The notion of a $k$-local representation $\lambda$ of a group $G$ over the field of complex numbers $\mathbb{C}$ can be defined in the same manner.
\end{remark}

Note that the concept of $k$-local representations can be extended to a group $G$ generated by $s(n-1)$ elements, where these generators are partitioned into $s$ distinct families, each consisting of $n-1$ generators. For simplicity, we introduce the case $s=2$ in the next definition.

\begin{definition} \cite{M.N1} 
Let $G$ be a group generated by the generators $g_1,g_2,\ldots,g_{n-1}$ and $h_1,h_2,\ldots,h_{n-1}$.  
A $k$-local representation $\lambda: G \longrightarrow \mathrm{GL}_{m}(\mathbb{Z}[t^{\pm 1}])$ is a representation such that
\[
\lambda(g_i)=
\begin{pmatrix}
I_{i-1} & 0 & 0 \\
0 & \left( M_i \right) & 0 \\
0 & 0 & I_{\,n-i-1}
\end{pmatrix}
\quad \text{and} \quad
\lambda(h_i)=
\begin{pmatrix}
I_{i-1} & 0 & 0 \\
0 & \left( N_i\right)  & 0 \\
0 & 0 & I_{\,n-i-1}
\end{pmatrix}
\]
for $1 \leq i \leq n-1$, where $M_i, N_i \in \mathrm{GL}_k(\mathbb{Z}[t^{\pm 1}])$, $k=m-n+2$, and $I_r$ denotes the $r \times r$ identity matrix. The representation $\lambda$ is said to be homogeneous if all the matrices $M_i$ are identical and all the matrices $N_i$ are identical.
\end{definition}

We now introduce two types of extension of $L_n$ to $SLM_n$ and $VSLM_n$. 
The first type, called the $k$-local type extension, applies to $k$-local representations of $L_n$. The second type, called the $\Phi$-type extension, applies to general representations of $L_n$ under certain conditions and is defined in a manner similar to that given by Bardakov et al. in \cite[Proposition 4.1]{Bar2024}. We start with the following definition of the first type.

\begin{definition}
Let $\lambda: L_n \longrightarrow \mathrm{GL}_{m}(\mathbb{Z}[t^{\pm 1}])$, where $t$ is an indeterminate, be a $k$-local representation. 
A $k$-local type extension of $\lambda$ to $SLM_n$ (respectively $VSLM_n$) is a $k$-local representation of $SLM_n$ (respectively $VSLM_n$) whose restriction to $L_n$ coincides with $\lambda$.
\end{definition}

We now proceed to the second type of extension. The following proposition explains how a representation of $L_n$ can be extended to $SLM_n$ through this type of extension.

\begin{proposition} \label{phitype}
Let $\lambda: L_n \longrightarrow \mathrm{GL}_{m}(\mathbb{Z}[t^{\pm 1}])$, where $t$ is an indeterminate, be a representation satisfying 
$\lambda(\ell_i)\lambda(\ell_j)=\lambda(\ell_j)\lambda(\ell_i)$, for all $|i-j|\geq 2$. 
Define a map $\Phi: SLM_n \longrightarrow \mathrm{M}_{m}(\mathbb{Z}[t^{\pm 1}])$ by
\[
\Phi(\ell_i)=\lambda(\ell_i)
\]
and
\[
\Phi(\tau_i)=a\lambda(\ell_i)+bI_m,
\]
where $a,b \in \mathbb{Z}[t^{\pm 1}]$. Then $\Phi$ defines a representation of $SLM_n$ that extends $\lambda$. Such an extension is called a $\Phi$-type extension.
\end{proposition}

\begin{proof}
To prove that $\Phi$ defines a representation of $SLM_n$, it suffices to verify that it preserves the relations \eqref{eqs17}--\eqref{eqs21}. We examine these relations one by one as follows.

\begin{itemize}

\item Consider the relation $\tau_i\tau_j=\tau_j\tau_i$ for $|i-j|\geq 2$. We compute
\[
\Phi(\tau_i)\Phi(\tau_j)
=(a\lambda(\ell_i)+bI_m)(a\lambda(\ell_j)+bI_m).
\]
Expanding the product yields
\[
\Phi(\tau_i)\Phi(\tau_j)
=a^2\lambda(\ell_i)\lambda(\ell_j)
+ab\,\lambda(\ell_i)
+ba\,\lambda(\ell_j)
+b^2I_m.
\]
Similarly, we compute
\[
\Phi(\tau_j)\Phi(\tau_i)
=(a\lambda(\ell_j)+bI_m)(a\lambda(\ell_i)+bI_m),
\]
and expanding this product gives
\[
\Phi(\tau_j)\Phi(\tau_i)
=a^2\lambda(\ell_j)\lambda(\ell_i)
+ab\,\lambda(\ell_j)
+ba\,\lambda(\ell_i)
+b^2I_m.
\]
Since $\lambda(\ell_i)\lambda(\ell_j)=\lambda(\ell_j)\lambda(\ell_i)$ whenever $|i-j|\geq 2$ by our assumption, it follows that
\[
\Phi(\tau_i)\Phi(\tau_j)=\Phi(\tau_j)\Phi(\tau_i).
\]
Thus, this relation is preserved.\vspace{0.1cm}

\item Consider the relation $\tau_i\ell_j=\ell_j\tau_i$ for $|i-j|\geq 2$. We compute
\[
\Phi(\tau_i)\Phi(\ell_j)=(a\lambda(\ell_i)+bI_m)\lambda(\ell_j),
\]
which gives
\[
\Phi(\tau_i)\Phi(\ell_j)=a\lambda(\ell_i)\lambda(\ell_j)+b\lambda(\ell_j).
\]
On the other hand, we compute
\[
\Phi(\ell_j)\Phi(\tau_i)=\lambda(\ell_j)(a\lambda(\ell_i)+bI_m),
\]
and expanding yields
\[
\Phi(\ell_j)\Phi(\tau_i)=a\lambda(\ell_j)\lambda(\ell_i)+b\lambda(\ell_j).
\]
Again, $\lambda(\ell_i)\lambda(\ell_j)=\lambda(\ell_j)\lambda(\ell_i)$ for $|i-j|\geq 2$ by our assumption, and so
\[
\Phi(\tau_i)\Phi(\ell_j)=\Phi(\ell_j)\Phi(\tau_i).
\]
Hence, this relation is also preserved.\vspace{0.1cm}
\item Consider the relation $\tau_i\ell_i=\ell_i\tau_i$ for $1\leq i\leq n-1$. We compute
\[
\Phi(\tau_i)\Phi(\ell_i)=(a\lambda(\ell_i)+bI_m)\lambda(\ell_i).
\]
Using the fact that $\lambda$ is a representation of $L_n$ and $\ell_i^2=1$, we obtain
\[
\Phi(\tau_i)\Phi(\ell_i)=aI_m+b\lambda(\ell_i).
\]
Similarly, we compute
\[
\Phi(\ell_i)\Phi(\tau_i)=\lambda(\ell_i)(a\lambda(\ell_i)+bI_m),
\]
which also gives
\[
\Phi(\ell_i)\Phi(\tau_i)=aI_m+b\lambda(\ell_i).
\]
Therefore,
\[
\Phi(\tau_i)\Phi(\ell_i)=\Phi(\ell_i)\Phi(\tau_i),
\]
and this relation also holds.\vspace{0.1cm}

\item Consider the relation $\ell_i\ell_{i+1}\tau_i=\tau_{i+1}\ell_i\ell_{i+1}$ for $1\leq i\leq n-2$. We compute
\[
\Phi(\ell_i)\Phi(\ell_{i+1})\Phi(\tau_i)
=\lambda(\ell_i)\lambda(\ell_{i+1})(a\lambda(\ell_i)+bI_m).
\]
Expanding this expression gives
\[
\Phi(\ell_i)\Phi(\ell_{i+1})\Phi(\tau_i)
=a\lambda(\ell_i)\lambda(\ell_{i+1})\lambda(\ell_i)
+b\lambda(\ell_i)\lambda(\ell_{i+1}).
\]
On the other hand, we compute
\[
\Phi(\tau_{i+1})\Phi(\ell_i)\Phi(\ell_{i+1})
=(a\lambda(\ell_{i+1})+bI_m)\lambda(\ell_i)\lambda(\ell_{i+1}),
\]
and expanding yields
\[
\Phi(\tau_{i+1})\Phi(\ell_i)\Phi(\ell_{i+1})
=a\lambda(\ell_{i+1})\lambda(\ell_i)\lambda(\ell_{i+1})
+b\lambda(\ell_i)\lambda(\ell_{i+1}).
\]
Since $\lambda$ is a representation and the generators of $L_n$ satisfy
\[
\ell_i\ell_{i+1}\ell_i=\ell_{i+1}\ell_i\ell_{i+1},
\]
it follows that
\[
\lambda(\ell_i)\lambda(\ell_{i+1})\lambda(\ell_i)
=
\lambda(\ell_{i+1})\lambda(\ell_i)\lambda(\ell_{i+1}).
\]
Consequently,
\[
\Phi(\ell_i)\Phi(\ell_{i+1})\Phi(\tau_i)
=
\Phi(\tau_{i+1})\Phi(\ell_i)\Phi(\ell_{i+1}),
\]
and the relation is preserved.
\end{itemize}
Remark that the relations \eqref{eqs20} and \eqref{eqs21} are equivalent by Remark \ref{RK6}, and this completes the proof.
\end{proof}

\begin{question}
Under what conditions on the parameters $a,b\in \mathbb{Z}[t^{\pm1
}]$ does the representation $\Phi$ of $SLM_n$ in Proposition \ref{phitype} become faithful?
\end{question}

The next proposition extends Proposition~\ref{phitype} to the virtual singular triplet monoid $VSLM_n$.

\begin{proposition} \label{phitype222}
Let $\lambda: L_n \longrightarrow \mathrm{GL}_{m}(\mathbb{Z}[t^{\pm 1}])$, where $t$ is an indeterminate, be a representation satisfying 
$\lambda(\ell_i)\lambda(\ell_j)=\lambda(\ell_j)\lambda(\ell_i)$ for all $|i-j|\geq 2$. 
Define a map $\Phi: VSLM_n \longrightarrow \mathrm{M}_{m}(\mathbb{Z}[t^{\pm 1}])$ by
\[
\Phi(\ell_i)=\lambda(\ell_i),
\]
\[
\Phi(\tau_i)=a\lambda(\ell_i)+bI_m,
\]
and 
\[
\Phi(\nu_i)=-\lambda(\ell_i),
\]
where $a,b \in \mathbb{Z}[t^{\pm 1}]$. Then $\Phi$ defines a representation of $VSLM_n$ that extends $\lambda$. Such an extension is also called a $\Phi$-type extension.
\end{proposition}

\begin{proof}
By Proposition \ref{phitype}, it remains to verify that $\Phi$ preserves the relations involving the generators $\nu_i$, namely \eqref{eqs22}--\eqref{eqs28}. We deal with these relations separately as follows. 
\begin{itemize}
\item Consider the relation $\nu_i^2=1$ for $1\leq i\leq n-1$. We compute
\[
\Phi(\nu_i)^2=(-\lambda(\ell_i))^2=\lambda(\ell_i)^2.
\]
Since $\lambda$ is a representation of $L_n$, we have $\lambda(\ell_i)^2=I_m$. Hence
\[
\Phi(\nu_i)^2=I_m,
\]
and so this relation is preserved.\vspace{0.1cm}

\item Consider the relation $\nu_i\nu_{i+1}\nu_i=\nu_{i+1}\nu_i\nu_{i+1}$ for $1\leq i\leq n-2$. We compute
\[
\Phi(\nu_i)\Phi(\nu_{i+1})\Phi(\nu_i)
=(-\lambda(\ell_i))(-\lambda(\ell_{i+1}))(-\lambda(\ell_i))
=-\lambda(\ell_i)\lambda(\ell_{i+1})\lambda(\ell_i).
\]
Similarly, we compute
\[
\Phi(\nu_{i+1})\Phi(\nu_i)\Phi(\nu_{i+1})
=(-\lambda(\ell_{i+1}))(-\lambda(\ell_i))(-\lambda(\ell_{i+1}))
=-\lambda(\ell_{i+1})\lambda(\ell_i)\lambda(\ell_{i+1}).
\]
Since $\lambda$ is a representation and the generators of $L_n$ satisfy
\[
\ell_i\ell_{i+1}\ell_i=\ell_{i+1}\ell_i\ell_{i+1},
\]
it follows that the two expressions are equal. Hence, this relation holds.\vspace{0.1cm}

\item Consider the relation $\nu_i\ell_{i+1}\nu_i=\nu_{i+1}\ell_i\nu_{i+1}$ for $1\leq i\leq n-2$. We compute
\[
\Phi(\nu_i)\Phi(\ell_{i+1})\Phi(\nu_i)
=(-\lambda(\ell_i))\lambda(\ell_{i+1})(-\lambda(\ell_i))
=\lambda(\ell_i)\lambda(\ell_{i+1})\lambda(\ell_i).
\]
On the other hand, we compute
\[
\Phi(\nu_{i+1})\Phi(\ell_i)\Phi(\nu_{i+1})
=(-\lambda(\ell_{i+1}))\lambda(\ell_i)(-\lambda(\ell_{i+1}))
=\lambda(\ell_{i+1})\lambda(\ell_i)\lambda(\ell_{i+1}).
\]
Using again the relations in $L_n$, the two expressions are equal and therefore, this relation is also preserved.\vspace{0.1cm}

\item Consider the relation $\nu_i\tau_{i+1}\nu_i=\nu_{i+1}\tau_i\nu_{i+1}$ for $1\leq i\leq n-2$. We compute
\[
\Phi(\nu_i)\Phi(\tau_{i+1})\Phi(\nu_i)
=(-\lambda(\ell_i))(a\lambda(\ell_{i+1})+bI_m)(-\lambda(\ell_i)).
\]
Expanding the product, setting $\lambda(\ell_i)^2=I_m$, gives
\[
\Phi(\nu_i)\Phi(\tau_{i+1})\Phi(\nu_i)
=a\lambda(\ell_i)\lambda(\ell_{i+1})\lambda(\ell_i)+bI_m.
\]
Similarly, we compute
\[
\Phi(\nu_{i+1})\Phi(\tau_i)\Phi(\nu_{i+1})
=(-\lambda(\ell_{i+1}))(a\lambda(\ell_i)+bI_m)(-\lambda(\ell_{i+1})),
\]
which yields
\[
\Phi(\nu_{i+1})\Phi(\tau_i)\Phi(\nu_{i+1})
=a\lambda(\ell_{i+1})\lambda(\ell_i)\lambda(\ell_{i+1})+bI_m.
\]
Again, using the relations in $L_n$, we find that the equality follows and the relation is preserved.\vspace{0.1cm}

\item Consider the relation $\nu_i\nu_j=\nu_j\nu_i$ for $|i-j|\geq 2$. We compute
\[
\Phi(\nu_i)\Phi(\nu_j)
=(-\lambda(\ell_i))(-\lambda(\ell_j))
=\lambda(\ell_i)\lambda(\ell_j)
\]
and
\[
\Phi(\nu_j)\Phi(\nu_i)
=(-\lambda(\ell_j))(-\lambda(\ell_i))
=\lambda(\ell_j)\lambda(\ell_i).
\]
By our assumption, we have $\lambda(\ell_i)\lambda(\ell_j)=\lambda(\ell_j)\lambda(\ell_i)$ for all $|i-j|\geq 2$, and so the equality holds, which implies that this relation is preserved.\vspace{0.1cm}

\item Consider the relation $\nu_i\ell_j=\ell_j\nu_i$ for $|i-j|\geq 2$. We compute
\[
\Phi(\nu_i)\Phi(\ell_j)
=(-\lambda(\ell_i))\lambda(\ell_j)
=-\lambda(\ell_i)\lambda(\ell_j)
\]
and
\[
\Phi(\ell_j)\Phi(\nu_i)
=\lambda(\ell_j)(-\lambda(\ell_i))
=-\lambda(\ell_j)\lambda(\ell_i).
\]
Again, by our assumption we have $\lambda(\ell_i)\lambda(\ell_j)=\lambda(\ell_j)\lambda(\ell_i)$ for all $|i-j|\geq 2$, which implies that the equality follows, and so the relation is preserved.\vspace{0.1cm}

\item Finally, consider the relation $\nu_i\tau_j=\tau_j\nu_i$ for $|i-j|\geq 2$. We compute
\[
\Phi(\nu_i)\Phi(\tau_j)
=(-\lambda(\ell_i))(a\lambda(\ell_j)+bI_m)
=-a\lambda(\ell_i)\lambda(\ell_j)-b\lambda(\ell_i)
\]
and
\[
\Phi(\tau_j)\Phi(\nu_i)
=(a\lambda(\ell_j)+bI_m)(-\lambda(\ell_i))
=-a\lambda(\ell_j)\lambda(\ell_i)-b\lambda(\ell_i).
\]
Since, by our assumption, we have  $\lambda(\ell_i)\lambda(\ell_j)=\lambda(\ell_j)\lambda(\ell_i)$ for all $|i-j|\geq 2$, we conclude that $
\Phi(\nu_i)\Phi(\tau_j)=\Phi(\tau_j)\Phi(\nu_i)$, and so this relation is also preserved.
\end{itemize}
Therefore, $\Phi$ preserves all the relations \eqref{eqs22}--\eqref{eqs28}. Consequently, $\Phi$ defines a representation of $VSLM_n$ extending $\lambda$.
\end{proof}


We end this section with the following corollary.

\begin{corollary}
Every $2$-local representation of $L_n$ admits an extension to $SLM_n$ and $VSLM_n$ via the $\Phi$-type extensions described in Propositions \ref{phitype} and \ref{phitype222}.
\end{corollary}

\begin{proof}
Let $\lambda$ be a $2$-local representation of $L_n$. By the block matrix form of $\lambda$, such a representation satisfies $
\lambda(\ell_i)\lambda(\ell_j)=\lambda(\ell_j)\lambda(\ell_i)$
for all $|i-j|\geq 2$. Therefore, the assumptions required in Propositions \ref{phitype} and \ref{phitype222} are satisfied. Consequently, $\lambda$ can be extended to representations of $SLM_n$ and $VSLM_n$ through the $\Phi$-type extensions described in those propositions.
\end{proof}

\section{On Extensions of the Representation $\mu$ of $L_n$ to $SLM_n$ and $VSLM_n$}\label{rep-extension_mu}

In this section, we continue to work with the presentation for the monoid $VSLM_n$ as given in Definition~\ref{monoid-original_presentation}. Our aim is to construct representations of $SLM_n$ and $VSLM_n$ over the ring of Laurent polynomials $\mathbb{Z}[t^{\pm 1}]$, where $t$ is an indeterminate. To this end, we consider a representation of $L_n$, namely $\mu: L_n \longrightarrow \mathrm{GL}_n(\mathbb{Z}[t^{\pm 1}])$, introduced in \cite{Nasser2026}, and investigate the possibility of extending it to $SLM_n$ and $VSLM_n$ via two types of extensions given in Section~\ref{rep-extensions}, and compare these extensions. We begin by recalling the definition of the representation $\mu$.

\begin{definition}\cite{Nasser2026}\label{defmu}
The representation $\mu: L_n \longrightarrow \mathrm{GL}_n(\mathbb{Z}[t^{\pm 1}])$, where $t$ is an indeterminate, is defined on the generators $\ell_i$ of $L_n$ as follows:
\[
\mu(\ell_i)=
\begin{pmatrix}
I_{i-1} & 0 & 0 \\
0 &
\begin{pmatrix}
0 & t^k \\
t^{-k} & 0
\end{pmatrix}
& 0 \\
0 & 0 & I_{n-i-1}
\end{pmatrix}, \quad 1\leq i \leq n-1.
\]
\end{definition}



\subsection{The Homogeneous Local Type Extension of $\mu$ to $SLM_n$ and $VSLM_n$}

In this subsection, we find all homogeneous $2$-local type extensions of the representation $\mu$ of $L_n$ to $SLM_n$ and $VSLM_n$ for all $n\geq 3$. The case $n=2$ is a very special case and is left as an exercise.

\begin{theorem} \label{th18}
Let $\mu':SLM_n \longrightarrow \mathrm{M}_n(\mathbb{Z}[t^{\pm 1}])$ be a homogeneous $2$-local representation extending $\mu$. Then $\mu'$ is uniquely determined and acts on the generators of $SLM_n$ as follows. For $1 \leq i \leq n-1$,
\[
\mu'(\ell_i)=
\begin{pmatrix}
I_{i-1} & 0 & 0 \\
0 &
\begin{pmatrix}
0 & t^k \\
t^{-k} & 0
\end{pmatrix}
& 0 \\
0 & 0 & I_{n-i-1}
\end{pmatrix}
\quad \text{and} \quad
\mu'(\tau_i)=
\begin{pmatrix}
I_{i-1} & 0 & 0 \\
0 &
\begin{pmatrix}
w & x \\
t^{-2k}x & w
\end{pmatrix}
& 0 \\
0 & 0 & I_{n-i-1}
\end{pmatrix},
\]
where $w,x \in \mathbb{Z}[t^{\pm 1}]$.
\end{theorem}

\begin{proof}
Since $\mu'$ extends $\mu$, we must have
\[
\mu'(\ell_i)=\mu(\ell_i)=
\begin{pmatrix}
I_{i-1} & 0 & 0 \\
0 &
\begin{pmatrix}
0 & t^k \\
t^{-k} & 0
\end{pmatrix}
& 0 \\
0 & 0 & I_{n-i-1}
\end{pmatrix},
\qquad 1\leq i\leq n-1 .
\]
Because $\mu'$ is a homogeneous $2$-local representation, we may write
\[
\mu'(\tau_i)=
\begin{pmatrix}
I_{i-1} & 0 & 0 \\
0 &
\begin{pmatrix}
w & x \\
y & z
\end{pmatrix}
& 0 \\
0 & 0 & I_{n-i-1}
\end{pmatrix}, \qquad 1\leq i \leq n-1,
\]
where $w,x,y,z \in \mathbb{Z}[t^{\pm 1}]$. Now, since $\mu'$ is a representation of $SLM_n$, it must preserve the defining relations of $SLM_n$. The relations involving $\ell_i$ and $\tau_i$ that are relevant here are
\[
\ell_1\tau_1=\tau_1\ell_1
\quad \text{and} \quad
\ell_1\ell_2\tau_1=\tau_2\ell_1\ell_2.
\]
Applying $\mu'$ to the first relation $\ell_1\tau_1=\tau_1\ell_1$ and comparing the corresponding matrix entries, we obtain
\begin{equation}
t^{-k}x - t^{k}y =0,
\end{equation}
\begin{equation}
t^{k}w - t^{k}z =0 .
\end{equation}
Similarly, applying $\mu'$ to the relation $\ell_1\ell_2\tau_1=\tau_2\ell_1\ell_2$ yields similar equations
\begin{equation}
-t^{-k}w + t^{-k}z =0,
\end{equation}
\begin{equation}
-t^{-k}x + t^{k}y =0 .
\end{equation}
Since $t$ is an indeterminate, the above system of four equations in $w,x,y,$ and $z$ implies that
\[
w=z \quad \text{and} \quad y=t^{-2k}x .
\]
Substituting these relations into $\mu'(\tau_i)$ gives that
\[
\mu'(\tau_i)=
\begin{pmatrix}
I_{i-1} & 0 & 0 \\
0 &
\begin{pmatrix}
w & x \\
t^{-2k}x & w
\end{pmatrix}
& 0 \\
0 & 0 & I_{n-i-1}
\end{pmatrix}, \qquad 1\leq i \leq n-1,
\]
where $w,x \in \mathbb{Z}[t^{\pm 1}]$ and this completes the proof.
\end{proof}

We now proceed to find all homogeneous $2$-local type extensions of $\mu$ to $VSLM_n$.

\begin{theorem} \label{th19}
Let $\mu'':VSLM_n \longrightarrow \mathrm{M}_n(\mathbb{Z}[t^{\pm 1}])$ be a homogeneous $2$-local representation extending $\mu$. Then $\mu''$ is uniquely determined and acts on the generators of $VSLM_n$ as follows. For $1 \leq i \leq n-1$,
$$
\mu''(\ell_i)=
\begin{pmatrix}
I_{i-1} & 0 & 0 \\
0 &
\begin{pmatrix}
0 & t^k \\
t^{-k} & 0
\end{pmatrix}
& 0 \\
0 & 0 & I_{n-i-1}
\end{pmatrix},
$$
$$
\mu''(\tau_i)=
\begin{pmatrix}
I_{i-1} & 0 & 0 \\
0 &
\begin{pmatrix}
w & x \\
t^{-2k}x & w
\end{pmatrix}
& 0 \\
0 & 0 & I_{n-i-1}
\end{pmatrix},
$$
and 
$$
\mu''(\nu_i)=
\begin{pmatrix}
I_{i-1} & 0 & 0 \\
0 &
\begin{pmatrix}
0 & q \\
\dfrac{1}{q} & 0
\end{pmatrix}
& 0 \\
0 & 0 & I_{n-i-1}
\end{pmatrix},
$$
where $w,x,q \in \mathbb{Z}[t^{\pm 1}]$.
\end{theorem}

\begin{proof}
Since $\mu''$ extends $\mu$, the images of the generators $\ell_i$ and $\tau_i$ must coincide with those obtained for $SLM_n$ in Theorem \ref{th18}. Hence, we have
\[
\mu''(\ell_i)=
\begin{pmatrix}
I_{i-1}&0&0\\
0&
\begin{pmatrix}
0&t^k\\
t^{-k}&0
\end{pmatrix}
&0\\
0&0&I_{n-i-1}
\end{pmatrix}, \qquad 1\leq i \leq n-1
\]
and
\[
\mu''(\tau_i)=
\begin{pmatrix}
I_{i-1}&0&0\\
0&
\begin{pmatrix}
w&x\\
t^{-2k}x&w
\end{pmatrix}
&0\\
0&0&I_{n-i-1}
\end{pmatrix}, \qquad 1\leq i \leq n-1,
\]
where $w,x\in\mathbb{Z}[t^{\pm1}]$. Also, since $\mu''$ is homogeneous $2$-local, the image of $\nu_i$ must also have a $2\times2$ block form. Hence we write
\[
\mu''(\nu_i)=
\begin{pmatrix}
I_{i-1}&0&0\\
0&
\begin{pmatrix}
p&q\\
r&s
\end{pmatrix}
&0\\
0&0&I_{n-i-1}
\end{pmatrix}, \qquad 1\leq i \leq n-1,
\]
where $p,q,r,s\in\mathbb{Z}[t^{\pm1}]$. Remark that in order to determine the possible values of $p, q, r$ and $s$, the only relations of $VSLM_n$ needed to be verified here are:
$$
\nu_1^2=1,\quad
\nu_1\nu_2\nu_1=\nu_2\nu_1\nu_2,\quad
\nu_1\ell_2\nu_1=\nu_2\ell_1\nu_2,\quad \text{and} \quad
\nu_1\tau_2\nu_1=\nu_2\tau_1\nu_2.
$$
First, consider the relation $\nu_1^2=1.$ Applying $\mu''$ gives
\[
\begin{pmatrix}
p&q\\
r&s
\end{pmatrix}^2=I_2,
\]
which yields
\begin{equation} \label{eqs36}
 p^2+qr=1,   
\end{equation}
\begin{equation}\label{eqs37}
 s^2+qr=1,   
\end{equation}
\begin{equation}\label{eqs38}
   (p+s)q=0,
\end{equation}
\begin{equation}\label{eqs39}
    (p+s)r=0.
\end{equation}
Next, consider the relation $
\nu_1\ell_2\nu_1=\nu_2\ell_1\nu_2 .$ Since the representation is homogeneous $2$-local, it suffices to examine the following $3\times3$ matrices. First, we write the matrices of the generators involved:
\[
\mu''(\ell_1)=
\begin{pmatrix}
0&t^k&0\\
t^{-k}&0&0\\
0&0&1
\end{pmatrix},
\qquad
\mu''(\ell_2)=
\begin{pmatrix}
1&0&0\\
0&0&t^k\\
0&t^{-k}&0
\end{pmatrix},
\]
\[
\mu''(\nu_1)=
\begin{pmatrix}
p&q&0\\
r&s&0\\
0&0&1
\end{pmatrix}, \quad \text{and}
\quad
\mu''(\nu_2)=
\begin{pmatrix}
1&0&0\\
0&p&q\\
0&r&s
\end{pmatrix}.
\]
We first compute
\[
\mu''(\nu_1)\mu''(\ell_2)\mu''(\nu_1)
=
\left(
\begin{array}{ccc}
 p^2 & p q & q t^k \\
 p r & q r & s t^k \\
 r t^{-k} & s t^{-k} & 0 \\
\end{array}
\right).
\]
Second, we compute
\[
\mu''(\nu_2)\mu''(\ell_1)\mu''(\nu_1)
=
\left(
\begin{array}{ccc}
 0 & p t^k & q t^k \\
 p t^{-k} & q r & q s \\
 r t^{-k} & r s & s^2 \\
\end{array}
\right).
\]
Comparing both sides directly gives $p=s=0$. Hence, by Equations \eqref{eqs36} and \eqref{eqs37}, we get $r=\dfrac{1}{q}$. 
Therefore
\[
\mu''(\nu_i)=
\begin{pmatrix}
I_{i-1}&0&0\\
0&
\begin{pmatrix}
0&q\\
\dfrac{1}{q}&0
\end{pmatrix}
&0\\
0&0&I_{n-i-1}
\end{pmatrix}, \qquad 1\leq i \leq n-1.
\]
Finally, one verifies that the remaining relations
\[
\nu_1\nu_2\nu_1=\nu_2\nu_1\nu_2
\quad\text{and}\quad
\nu_1\tau_2\nu_1=\nu_2\tau_1\nu_2
\]
are preserved under $\mu''$ with these matrices and do not impose further conditions. Hence, the form of $\mu''(\nu_i)$ is uniquely determined, which completes the proof.
\end{proof}

\subsection{The $\Phi$-Type Extension of $\mu$ to $SLM_n$ and $VSLM_n$}

In this subsection, we determine all $\Phi$-type extensions of the representation $\mu$ of $L_n$ to $SLM_n$ and $VSLM_n$ for all $n\geq 3$. We begin with the case of $SLM_n$.

\begin{proposition} \label{prop123}
Consider the representation $\mu: L_n \longrightarrow \mathrm{GL}_n(\mathbb{Z}[t^{\pm 1}])$ given in Definition~\ref{defmu}. The $\Phi$-type extension of $\mu$ to $SLM_n$ is the representation $
\Phi_1:SLM_n \longrightarrow \mathrm{M}_n(\mathbb{Z}[t^{\pm 1}])$
given by
\[
\Phi_1(\ell_i)=
\begin{pmatrix}
I_{i-1} & 0 & 0 \\
0 &
\begin{pmatrix}
0 & t^k \\
t^{-k} & 0
\end{pmatrix}
& 0 \\
0 & 0 & I_{n-i-1}
\end{pmatrix}, \qquad 1\leq i \leq n-1
\]
and
\[
\Phi_1(\tau_i)=
\begin{pmatrix}
(a+b) I_{i-1} & 0 & 0 \\
0 &
\begin{pmatrix}
b & a t^k \\
a t^{-k} & b
\end{pmatrix}
& 0 \\
0 & 0 & (a+b) I_{n-i-1}
\end{pmatrix}, \qquad 1\leq i \leq n-1,
\]
where $a,b \in \mathbb{Z}[t^{\pm 1}]$.
\end{proposition}

\begin{proof}
The result follows by direct computation using the construction of $\Phi_1$ given in Proposition \ref{phitype}. Indeed, we have, for all $1\leq i \leq n-1,$
\[
\Phi_1(\ell_i)=\mu(\ell_i),
\]
and
\[
\Phi_1(\tau_i)=a\,\mu(\ell_i)+bI_n,
\]
where $a,b \in \mathbb{Z}[t^{\pm 1}]$, from which the matrices above are obtained.
\end{proof}

\begin{corollary}
The homogeneous $2$-local type extension of $\mu$, denoted by $\mu'$ in Theorem~\ref{th18}, and the $\Phi$-type extension of $\mu$, denoted by $\Phi_1$ in Proposition~\ref{prop123}, coincide as representations of $SLM_n$ whenever the parameters in $\mathbb{Z}[t^{\pm 1}]$ satisfy
\[
a+b=1, \quad b=w, \quad \text{and} \quad  x=a t^k .
\]
\end{corollary}

\begin{proof}
    The proof follows directly by comparing the representations $\mu'$ and $\Phi_1$.
\end{proof}

We now find all $\Phi$-type extensions of $\mu$ to $VSLM_n$.

\begin{proposition} \label{prop456}
Consider the representation $\mu: L_n \longrightarrow \mathrm{GL}_n(\mathbb{Z}[t^{\pm 1}])$ given in Definition~\ref{defmu}. The $\Phi$-type extension of $\mu$ to $VSLM_n$ is the representation $
\Phi_2:VSLM_n \longrightarrow \mathrm{M}_n(\mathbb{Z}[t^{\pm 1}])$
given by,
\[
\Phi_2(\ell_i)=
\begin{pmatrix}
I_{i-1} & 0 & 0 \\
0 &
\begin{pmatrix}
0 & t^k \\
t^{-k} & 0
\end{pmatrix}
& 0 \\
0 & 0 & I_{n-i-1}
\end{pmatrix}, \qquad 1\leq i \leq n-1,
\]
\[
\Phi_2(\tau_i)=
\begin{pmatrix}
(a+b) I_{i-1} & 0 & 0 \\
0 &
\begin{pmatrix}
b & a t^k \\
a t^{-k} & b
\end{pmatrix}
& 0 \\
0 & 0 & (a+b) I_{n-i-1}
\end{pmatrix}, \qquad 1\leq i \leq n-1,
\]
and 
\[
\Phi_2(\nu_i)=
\begin{pmatrix}
-I_{i-1} & 0 & 0 \\
0 &
\begin{pmatrix}
0 & -t^k \\
-t^{-k} & 0
\end{pmatrix}
& 0 \\
0 & 0 & -I_{n-i-1}
\end{pmatrix}, \qquad 1\leq i \leq n-1,
\]
where $a,b \in \mathbb{Z}[t^{\pm 1}]$.
\end{proposition}

\begin{proof}
The result follows by direct computation using the construction of $\Phi_2$ given in Proposition~\ref{phitype222}. Indeed, we have, for all $1\leq i \leq n-1,$
$$
\Phi_2(\ell_i)=\mu(\ell_i),
$$
\[
\Phi_2(\tau_i)=a\,\mu(\ell_i)+bI_n,
\]
and
\[
\Phi_2(\nu_i)=-\mu(\ell_i),
\]
where $a,b \in \mathbb{Z}[t^{\pm 1}]$, from which the matrices above are obtained.
\end{proof}

\begin{corollary}
The homogeneous $2$-local type extension of $\mu$, denoted by $\mu''$ in Theorem~\ref{th19}, and the $\Phi$-type extension of $\mu$, denoted by $\Phi_2$ in Proposition~\ref{prop456}, do not coincide as representations of $VSLM_n$, regardless of the choice of parameters in $\mathbb{Z}[t^{\pm 1}]$.
\end{corollary}

\begin{proof}
Since $\mu''(\nu_i) \neq \Phi_2(\nu_i)$ for all $q \in \mathbb{Z}[t^{\pm 1}]$, the two representations cannot coincide.
\end{proof}

We conclude this section with the following question.

\begin{question}
Is it possible to construct a new type of extension of representations of $L_n$ to $SLM_n$ and $VSLM_n$ without imposing any restrictions on $L_n$?
\end{question}

\section{Future Ideas}

We present several directions for future research related to the results obtained in this paper, in addition to the questions raised earlier in Sections~\ref{rep-extensions} and~\ref{rep-extension_mu}.

\begin{itemize}
    \item[(1)] What is the general relationship between the $k$-local type extensions and the $\Phi$-type extensions of representations of $L_n$ introduced in this paper?
    
    \item[(2)] Can these two types of extensions be generalized to other algebraic structures related to $L_n$ and its extensions?
    
    \item[(3)] Since the representation $\mu$ was shown to be reducible in \cite{Nasser2026}, what can be said about its extensions under the two types considered in this paper? In particular, are these extensions also reducible?
\end{itemize}

These questions suggest several promising directions for further investigation and may contribute to a deeper understanding of representations of $L_n$ and its associated algebraic structures.

\textbf{Acknowledgments.} C. Caprau was partially suported by NSF-RUI grant DMS 2204386.

\end{document}